\documentclass[11pt, reqno]{amsart}
\usepackage{amssymb,latexsym,amsmath,amsfonts}
\usepackage{latexsym}
\usepackage{color}
\usepackage[mathscr]{eucal}

\numberwithin{equation}{section}

\theoremstyle{definition}
\newtheorem{definition}{Definition}[section]
\newtheorem{example}[definition]{Example}

\theoremstyle{remark}
\newtheorem{remark}[definition]{Remark}
\theoremstyle{plain}
\newtheorem{theorem}[definition]{Theorem}
\newtheorem{lemma}[definition]{Lemma}

\newtheorem{result}[definition]{Result}

\newcommand{\eps}{\varepsilon}

\newcommand{\zbar}{\overline{z}}

\newcommand{\tht}{\theta}

\newcommand{\om}{\omega}

\newcommand{\bdy}{\partial}                                           
\newcommand{\OM}{\Omega}

\newcommand{\dsc}{\mathbb{D}}

\newcommand{\smoo}{\mathcal{C}}
\newcommand{\hol}{\mathcal{O}}
\newcommand{\poly}{\mathscr{P}}

\newcommand{\bcdot}{\boldsymbol{\cdot}}
\newcommand{\rl}{{\sf Re}}
\newcommand{\imag}{{\sf Im}}
\newcommand{\impl}{\Longrightarrow}

\newcommand{\lrarw}{\longrightarrow}

\newcommand\ba[1]{\overline{#1}}
\newcommand\hull[1]{\widehat{#1}}

\newcommand{\CC}{\mathbb{C}^2}
\newcommand{\cplx}{\mathbb{C}}

\newcommand{\rea}{\mathbb{R}}

\begin{document}
\title[On polynomial convexity of totally-real submanifolds]{On polynomial convexity
of compact subsets of totally-real submanifolds in $\cplx^n$}
\author{Sushil Gorai}
\address{Statistics and Mathematics Unit, Indian Statistical Institute Bangalore,
R. V. C. E. Post, Bangalore -- 560 059}
\email{sushil.gorai@gmail.com}
\thanks{This work is supported by an INSPIRE Faculty Fellowship (IFA-11MA-02) funded by DST}
\begin{abstract}
Let $K$ be a compact subset of a
totally-real manifold $M$, where 
$M$ is either a $\smoo^2$-smooth graph in $\cplx^{2n}$ over $\cplx^n$, or $M=u^{-1}\{0\}$ for 
 a $\smoo^2$-smooth submersion $u$ from $\cplx^n$ to $\rea^{2n-k}$, $k\leq n$.
In this case we show that $K$ is polynomially convex if and only if 
for a fixed neighbourhood $U$, defined in terms of the defining functions 
of $M$, there exists a plurisubharmonic function $\Psi$ on $\cplx^n$ such that 
$K\subset \{\Psi<0\}\subset U$.
\end{abstract}
\maketitle

 \section{Introduction and statements of the results}\label{S:intro}
The {\em polynomially convex hull} of a compact subset $K$ of $\cplx^n$  is defined as
 $\hull{K}:=\{z\in \cplx^n : |p(z)|\leq \sup_K|p|, p\in \cplx[z_1,\dots,z_n]\}$. We say that 
 $K$ is {\em polynomially convex} if $\hull{K}=K$. 
 As a motivation for studying polynomial convexity, we discuss briefly some of its connections 
 with the theory of uniform approximation by polynomials. 
Let $\poly(K)$ denote the uniform algebra on $K$ generated 
 by holomorphic polynomials. A fundamental question in the theory of uniform algebras 
 is to characterize the compacts $K$ of $\cplx^n$ for which 
 \begin{equation}\label{E:approx}
  \poly(K)=\smoo(K),
 \end{equation}
 where $\smoo(K)$ is the class of all continuous functions on $K$.
For $K\subset \rea\subset \cplx$, \eqref{E:approx} follows from Stone-Weierstrass theorem. More generally,
Lavrentiev \cite{Lv} showed 
that {\em $K\subset\cplx$ has Property \eqref{E:approx} if and only if 
$K$ is polynomially convex and has empty interior.}
In contrast, no such characterization is available for compact subsets of $\cplx^n$, $n\geq 2$. 
 Since the maximal ideal space of $\poly(K)$, $K\subset \cplx^n$,  
is identified with $\hull{K}$  via Gelfand's theory of commutative Banach algebras (see \cite{Gamelin} for details),  we observe that 
\[
\poly(K)=\smoo(K)\impl \hull{K}=K.
\]
With the assumption that $K$ is polynomial convex, there are several results, for instance see \cite{AIW1, AIW2, OPW, St1, W1},
that describe situations when \eqref{E:approx} holds.
Unless there is some way to determine whether $K\subset \cplx^n$, $n\geq 2$, is polynomially convex---which, in general,  is very 
difficult to determine---all of these results 
are somewhat abstract.
One such result is due to O'Farrell, Preskenis and Walsh \cite{OPW} which, in essence, says that polynomial convexity is 
sufficient for certain classes of compact subsets of $\cplx^n$ to satisfy Property~\eqref{E:approx}. More precisely:
 \begin{result}[O'Farrell, Preskenis and Walsh]\label{R:OPW}
 Let K be a compact polynomially convex subset of $\cplx^n$. Assume that $E$ is a closed subset of $K$ such that 
 $K\setminus E$ is locally contained in totally-real manifold. Then 
 \[
  \poly(K)=\{f\in \smoo(K): f|_E\in \poly(E)\}.
 \]
\end{result}
By Result~\ref{R:OPW}, if  $K$ is a compact polynomially convex subset of a totally-real submanifold of $\cplx^n$, then 
$\poly(K)=\smoo(K)$. In view of this fact, one is motivated to focus---with the goal of polynomial convexity---on 
characterizing the class of compact subsets of $\cplx^n$ that lie locally in 
some totally-real submanifold of $\cplx^n$. 
\smallskip

A totally-real set $M$ of $\cplx^n$ is locally 
polynomially convex at each $p\in M$, i.e., for each point $p\in M$ there exists 
a ball $B(p,r)$ in $\cplx^n$ such that $M\cap \ba{B(p,r)}$ is polynomially convex (see \cite{W2} for a proof in $\CC$ and \cite{HW, HerW} 
for a proof in $\cplx^n$, $n\geq 2$).
In general, an arbitrary compact subset of a 
 totally-real submanifold in $\cplx^n$ is not necessarily polynomially convex, as shown by the following example due to
 Wermer \cite[Example 6.1]{HW}:
 let \[M:=\{(z, f(z))\in \CC : z\in \cplx\},\] where 
 \[f(z)=-(1+i)\ba{z}+iz\ba{z}^2+z^2\ba{z}^3.\]
 It is easy 
  to see that $M$ is totally-real.
  Consider the compact subset 
  $K:=\{(z, f(z))\in \CC : z\in \ba{\dsc}\}\subset M.$
    Since $f(e^{i\tht})=0$ for $\tht\in \rea$, by using maximum modulus theorem, we infer that 
  $\hull{K}$ contains the analytic disc $\{(z,0)\in \CC : z\in \ba{\dsc}\}$. Hence, $K$ is not polynomially 
  convex. Some sufficient conditions 
  for polynomial convexity of totally-real discs in $\cplx^2$, i.e., the compact subset $\{(z,f(z))\in \cplx^2 ; z\in \ba{\dsc}\}$ of a totally 
  real graph in $\cplx^2$, in terms of the graphing function $f$, are available in the literature (see \cite{duval, OP1,OP2}, 
  and \cite{sangar} for a nice survey), but there are no general results 
 for compact subsets of $\cplx^n$, $n>2$, that we are aware of. 
  Therefore, it seems interesting to know  the conditions under which a compact subset of a totally-real submanifold of $\cplx^n$ 
 is polynomially convex. In this paper we report the results of our investigations on this question. 
 \smallskip
 
 We now present the main results of this paper. But first we state a lemma, which has a vital role in the proofs of our theorems,  about 
 the polynomially convex hull of general compact 
 subsets $\cplx^n$, which might also be of independent interest. 

\begin{lemma} \label{P:extPsh}
Let $K$ be a compact set in $\cplx^n,~n\geq 2$, and let $\phi$ be a plurisubharmonic function on $\cplx^n$ 
such that $K \subset \OM$, where $\OM:= \{z\in \cplx^n : \phi(z)< 0 \}$. Suppose there exists a non-negative 
function $v \in {\sf psh}(\OM)$ such that $v(z)=0~~\forall z \in K$. Then $\hull{K} \subseteq v^{-1}\{0\}$.
\end{lemma}

\noindent We know, from H\"{o}rmander's result \cite[Theorem~4.3.3]{H} (see Result~\ref{R:hormander} in 
Section~\ref{S:technical}), that the above is true when 
$\OM=\cplx^n$. Extending the plurisubharmonic 
function $v$ from $\OM$ to $\cplx^n$ is the key to our proof of Lemma~\ref{P:extPsh}, which 
we do by using a result of Poletsky \cite{P} (see Section~\ref{S:technical} 
for details) .
\smallskip


Let $K$ be a compact subset of a totally-real graph over $\cplx^n$ in $\cplx^{2n}$. 
In this case we present a necessary and sufficient condition for polynomial 
convexity of the given compact $K$ in terms of the graphing functions:
\begin{theorem}\label{T:totreal_PCVX}
 Let $f^1,\dots, f^n: \cplx^n \lrarw \cplx$ be $\smoo^2$-smooth functions such that, writing 
$F=(f^1, \dots, f^n)$, the graph ${\sf Gr}_{\cplx^n}(F)$ is a totally real submanifold of $\cplx^{2n}$. Then, a compact subset 
$K$ of ${\sf Gr}_{\cplx^n}(F)$ is polynomially convex if and only if there exists a 
$\Psi \in {\sf psh}(\cplx^{2n})$ such that
\[
 K \subset \om \subset \left\lbrace (z,w)\in \cplx^n \times \cplx^n: 
\sum_{\nu=1}^n |f^{\nu}(z)-w_{\nu}| < \frac{m(z)}{2L(z)} \right\rbrace,
\]
where 
\begin{align}
\om&:=\{(z,w)\in \cplx^n \times \cplx^n : \Psi(z,w)<0\},\notag\\
 L(z)&:= \max_{\nu \leq n}\left[\sup_{||v||=1}|\mathfrak{L}f^{\nu}(z;v)|\right]; \quad\text{and} \notag\\
m(z)&:= \inf_{||v||=1} \left(\sum_{\nu=1}^n|f_{\ba{z_1}}^{\nu}(z) \ba{v_1}+ \dots + f_{\ba{z_n}}^{\nu}(z) \ba{v_n}|^2 \right) \notag.
\end{align}
\end{theorem}
\noindent Here, and in what follows, $\mathfrak{L}f(z;.)$ denotes the Levi-form of a $\smoo^2$-smooth function $f$
at $z$. We now make a couple of remarks that will aid the understanding of the statement of  Theorem~\ref{T:totreal_PCVX}.

\begin{remark}\label{rem-tube}
The radius of the tube-like set 
\[\left\lbrace (z,w)\in \cplx^n \times \cplx^n: 
\sum_{\nu=1}^n |f^{\nu}(z)-w_{\nu}| < \frac{m(z)}{2L(z)} \right\rbrace\] 
may vary pointwise in $\cplx^{2n}$ but
the totally-real assumption on the graph ${\sf Gr}_{\cplx^n}(F)$  ensures that $m(z)\neq 0$ for $z\in \cplx^n$ 
(see Lemma~\ref{L:totreal} in Section~\ref{S:technical}).  
Therefore, the above tube-like set is a nonempty open subset of $\cplx^{2n}$ containing the compact $K$.
\end{remark}

\begin{remark}
We observe that if, in addition, we assume that the functions $f_1, \dots, f_n$  in Theorem~\ref{T:totreal_PCVX} are pluriharmonic, 
then the above tubular neighbourhood has infinite radius 
 at each point of ${\sf Gr}_{\cplx^n}(F)$. We just choose $\omega$ to be 
 a suitable polydisc containing $K$, $K\subset {\sf Gr}_{\cplx^n}(F)$,  
 such that the conditions of Theorem~\ref{T:totreal_PCVX} are satisfied.
 Thus, any compact subset of such a graph is 
 polynomially convex.
\end{remark}

\noindent We would like to mention that Theorem~\ref{T:totreal_PCVX} is a generalization of a 
result \cite[Theorem 6.1]{SG}---which characterizes 
polynomial convexity of graphs over polynomially convex subset $\ba{\OM}$, where $\OM$ is a bounded domain in $\cplx^n$---in author's 
dissertation. 
\smallskip

We now consider the case when the compact $K$ lies 
in a totally-real submanifold which is a level set of a $\smoo^2$-smooth submersion on $\cplx^n$. 
\begin{theorem}\label{T:totrlpoly}
 Let $M$ be a $\smoo^2$-smooth totally-real submanifold 
of $\cplx^n$ of real dimension $k$ such that $M:=\rho^{-1}\{0\}$, where $\rho:=(\rho_1,\dots,\rho_{2n-k})$ 
is a submersion from $\cplx^n$ to $\rea^{2n-k}$, and
 $K$ is a compact subset of $M$. Then $K$ is 
 polynomially convex if and only if there exists $\Psi\in {\sf psh}(\cplx^n)$ such that 
 \[
  K\subset \omega \subset \left\lbrace z\in \cplx^n : \sum_{l=1}^{2n-k}|\rho_l(z)|<\dfrac{m(z)}{L(z)} \right\rbrace,
 \]
 where 
 \begin{align}
 \omega&:=\{z\in \cplx^n : \Psi(z)<0\},\notag\\
L(z)&:=\max_{l\leq 2n-k} \left(\sup_{||v||=1}|\mathfrak{L}\rho_l(z, v)|\right);\quad \text{and}\notag\\
 m(z)&:=\inf_{||v||=1}\sum_{l=1}^{2n-k} \left|\sum_{j=1}^n\bdy_{\ba{z_j}} \rho_l (z)v_j\right|^2.\notag
 \end{align}
\end{theorem}

\begin{remark}
It is well known that a compact subset $K\subset \cplx^n$ 
is polynomially convex if and only if, for {\em every} neighbourhood $U$, there exists a polynomial 
polyhedron that contains the compact  and lies inside $U$. From Theorem~\ref{T:totrlpoly} 
 we conclude that  for a compact subset $K$ of a 
totally-real submanifold of $\cplx^n$ to be polynomially 
convex it suffices that, for {\em a single} fixed neighbourhood $U$ depending on the defining 
equations, we can find a polynomial polyhedron that contains $K$ and is contained in $U$. 
\end{remark}


As in Remark~\ref{rem-tube}, the fact that $\rho^{-1}\{0\}$ is totally real ensures that $m(z)\neq 0$, for 
all $z\in K$ (see Lemma~\ref{L:chartotrl} in Section~\ref{S:technical}); the set 
$\left\lbrace z\in \cplx^n : \sum_{l=1}^{2n-k}|\rho_l(z)|<\dfrac{m(z)}{L(z)} \right\rbrace$ is 
an open set containing $K$.
\smallskip

 We conclude the section with an observation about polynomial convexity
of compact subsets that lie in an arbitrary totally-real submanifold of $\cplx^n$, and not \emph{just} a zero set of 
a submersion defined on \emph{all of} $\cplx^n$.  
\begin{remark}
An abstract result analogous to Theorem~\ref{T:totrlpoly} holds for compacts that lie in any arbitrary totally-real 
submanifold of $\cplx^n$. The construction of a suitable tubular 
neighbourhood that will replace the tube-like neighbourhood in Theorem~\ref{T:totrlpoly}  is the main obstacle, 
which can be overcome by using partitions of unity. In this case, locally, we have real valued 
$\smoo^2$-smooth functions $\rho_1,\dots,\rho_{2n-k}$ such that the submanifold can be viewed locally 
as the zero set of a submersion $\rho=(\rho_1,\dots,\rho_{2n-k})$;
thus, locally we get a neighbourhood defined in terms of $\rho_1,\dots, \rho_{2n-k}$ as in Theorem~\ref{T:totrlpoly}. 
The problem with the result that we will end up with is that, since the tube $\omega$ would be given in terms of 
local data and (highly non-unique) cut-off functions, it would be merely an abstraction. Of course, highly abstract 
characterisations of polynomial convexity, in the language of uniform algebras, \emph{already} exist---but hard to 
check. The point of this paper is to begin with some natural overarching assumption and derive characterisations for polynomial 
convexity that are checkable. A couple of examples of totally-real submanifolds are 
given in Section~\ref{S:example} as applications of Theorem~\ref{T:totreal_PCVX} 
and Theorem~\ref{T:totrlpoly}.
\end{remark}

\section{Technical Results}\label{S:technical}
In this section, we prove some results that will be used in the proofs of our theorems: Lemma~\ref{P:extPsh}, a result about 
closed subsets of polynomially convex compact sets (Lemma~\ref{L:totrlsub}), and two results 
characterizing when a submanifold of $\cplx^n$ is totally real 
(Lemma~\ref{L:totreal} and Lemma~\ref{L:chartotrl}). 
We begin with the proof of Lemma~\ref{P:extPsh}. For that we need a result  
by H\"{o}rmander~\cite[Theorem~4.3.4]{H} that will be used several times in this paper. 

\begin{result}[H\"{o}rmander, Lemma 4.3.4, \cite{H}]\label{R:hormander}
Let $K$ be a compact subset of a pseudoconvex open set $\OM \subset \cplx^n$. Then
$\hull{K}_{\OM} = \hull{K}_{\OM}^P $, where $\hull{K}_{\OM}^P:=\left\lbrace z \in \cplx^n: u(z) \leq 
 \sup\nolimits_{K} u \;\forall u \in {\sf psh}(\OM) \right\rbrace$ and $\hull{K}_\OM:= 
 \left\lbrace z\in \cplx^n : |f(z)|\leq \sup\nolimits_{z\in K} |f(z)|\;\forall f\in \hol(\OM)\right\rbrace$.
\end{result}

\noindent We note that if $\OM= \cplx^n$ then Result \ref{R:hormander} says that 
the polynomially convex hull of $K$ is equal to the plurisubharmonically convex hull of $K$. 
\smallskip

We also need a result by Poletsky \cite[Lemma~4.1]{P} for proving Lemma~\ref{P:extPsh} that gives 
a sufficient condition for extending a plurisubharmonic function to $\cplx^n$. 

\begin{result}[Poletsky, Lemma 4.1, \cite{P}]\label{R:poletsky}
Let $K$ be a compact subset of an open set $V\subseteq {\cplx}^n$ and assume that there is a continuous 
plurisubharmonic function $u$ on $V$ such that $u=0$ on $K$ and positive on $V \setminus K$. If $v$ 
is a plurisubharmonic function defined on a neighbourhood $W \subset V$ of $K$ and bounded below on $K$, 
then there exists a plurisubharmonic function $v'$ on $V$ which coincides with $v$ on $K$.
\end{result}

We are now in a position to present the proof of Proposition~\ref{P:extPsh}.
\begin{proof}[Proof of Proposition~\ref{P:extPsh}]
We are given that $K \Subset \OM=\{z\in \cplx^n : \phi(z)<0\}$. Hence, by the Result \ref{R:hormander},
 $\hull{K} \subset \OM$. Upper-semicontinuity of $\phi$ gives $\hull{K} \Subset \OM$. Since $\hull{K}$ 
is polynomially convex, it follows from a result of Catlin \cite{C} (see \cite[Proposition~1.3]{S} also)
that there exists a continuous plurisubharmonic function $u$ on $\cplx^n$ such that $u=0$ on $\hull{K}$ and $u>0$ on
$\cplx^n \setminus \hull{K}$. Since, by hypothesis, $v \in {\sf psh}(\OM)$ is bounded below
on $\hull{K}$, the conditions of Result~\ref{R:poletsky} are fulfilled with $W:= \OM$ and $V:= \cplx^n$. Hence, 
there exists a plurisubharmonic function $v'$ on $\cplx^n$ such that
\begin{equation} \label{E:extnpsh}
v'(z)=v(z) \quad \forall z \in \hull{K}.
\end{equation}
Now, in view of the fact that $v=0$ on $K$ and $K \subset \hull{K}$, $v'(z)=0 \; \forall z \in K$.
At this point, using Result \ref{R:hormander} of H\"{o}rmander,  
 it follows that 
\[
 v'(z) \ \leq \ \sup_{z \in K} v'(z)=0 \quad \forall z \in \hull{K}.
\]
Thus, in view of \eqref{E:extnpsh}, we get
\[
 v(z) \ \leq \ 0 \quad\forall z \in \hull{K}.
\]
Since $v\geq 0$ (by hypothesis), we have $v(z)=0$ for all $z \in \hull{K}$. 
\end{proof}

\begin{lemma}\label{L:totrlsub}
 Let $K$ be a compact polynomially convex subset of a totally-real submanifold of $\cplx^n$. Then any 
 closed subset of $K$ is polynomially convex.  
\end{lemma}
\begin{proof}
Since $K$ is a polynomially convex subset of a totally-real submanifold of $\cplx^n$, 
 we apply Result~\ref{R:OPW} to get $\poly(K)=\smoo(K)$. 
 Let $L$ be a closed subset of $K$. By Tietze extension theorem, 
 $\smoo(K)|_L=\smoo(L)$. Since $\poly(K)|_L\subset \poly(L)\subset \smoo(L)$, we have $\poly(L)=\smoo(L)$. 
 Hence, $L$ is polynomially convex.
\end{proof}

We now state a result due to Oka (see \cite[Lemma 2.7.4]{H}) that 
gives us one direction of the implications in both the theorems in this paper.
\begin{result}[Oka]\label{R:polyhed}
 Let $K$ be a compact polynomially convex set in $\cplx^n$ and let $U$ be a 
 neighbourhood of $K$. Then there exist finitely many polynomials $p_1, \dots, p_m$ such that 
 \[
  K\subset \{z\in \cplx^n : |p_j(z)|\leq 1,\; j=1, \dots, m\} \subset U.
 \]
\end{result}

Let $M$ be a $\smoo^2$-smooth real submanifold of $\cplx^n$ of real dimension $k$, $k\leq n$. For each $p\in M$ 
there exists a neighbourhood $U_p$ of $p$ in $\cplx^n$ and $\smoo^2$-smooth real-valued functions 
$\rho_1,\rho_2, \dots, \rho_{2n-k}$ such that 
\[
 U_p\cap M=\{z\in U_p: \rho_1(z)=\rho_2(z)=\dots=\rho_{2n-k}(z)=0\},
\]
where $\rho=(\rho_1,\dots,\rho_{2n-k})$ is a submersion.
With these notations we now state the following lemma.
\begin{lemma}\label{L:chartotrl}
$M$ is totally real at $p\in M$ 
if and only if the matrix $A_p$ is of rank $n$, where
\[
A_p:=\begin{pmatrix}
  \bdy_{\ba{z_1}}\rho_1(p)\;\; \bdy_{\ba{z_2}}\rho_1(p)\;\; \cdots \bdy_{\ba{z_n}}\rho_1(p)\\
  \bdy_{\ba{z_1}}\rho_2(p)\;\; \bdy_{\ba{z_2}}\rho_2(p)\;\; \cdots \bdy_{\ba{z_n}}\rho_2(p)\\
  \vdots \\
  \bdy_{\ba{z_1}}\rho_{2n-k}(p)\;\; \bdy_{z_2}\rho_{2n-k}(p)\;\; \cdots \bdy_{\ba{z_n}}\rho_{2n-k}(p)
  \end{pmatrix}.
  \]
\end{lemma}
 
\begin{proof}
Viewing $\cplx^n$ as $\rea^{2n}$, the tangent space $T_pM$ can be described as: 
\[
 T_pM=\left\lbrace v \in \rea^{2n} :  
 D\rho(p)v=0
 \right\rbrace.
\]

We first assume that $M$ is totally real at $p\in M$. We will show that the rank of $A_p$ is $n$.
 Suppose the matrix $A_p$ has rank less than $n$. Without loss of generality, we may 
 assume that the rank of $A_p$ is $n-1$. Hence, there exists 
 $v=(v_1,\dots,v_n)\in \cplx^n\setminus \{0\}$ such that 
 \[
  A_pv=0.
 \]
This implies that the system of linear equation 
 \begin{align}
  \sum_{j=1}^n\dfrac{\bdy \rho_l}{\bdy \ba{z_j}}(p)v_j = & 0, \;l=1,\dots,{2n-k}, \label{E:sys}
 \end{align}
has a nonzero solution. Viewing $v_j=v'_j+iv_j^{''}$, $j=1,\dots, n$, and writing 
the system of equations \eqref{E:sys} in terms 
of real coordinates,
we obtain: for each $l=1,\dots, 2n-k$, 
\begin{align}
 \sum_{j=1}^n\left(\dfrac{\bdy \rho_l}{\bdy x_j}(p)+i\dfrac{\bdy \rho_l}{\bdy y_j}(p)\right)(v'_j+iv''_j)& 
 =0\notag\\
 \Longleftrightarrow \sum_{j=1}^n\left(\dfrac{\bdy \rho_l}{\bdy x_j}(p)v'_j
 -\dfrac{\bdy \rho_l}{\bdy y_j}(p)v''_j\right) & =0;  \label{E:ivtan} \\
 \qquad\qquad \sum_{j=1}^n\left(\dfrac{\bdy \rho_l}{\bdy x_j}(p)v''_j
 +\dfrac{\bdy \rho_l}{\bdy y_j}(p)v'_j \right)& =0.\label{E:vtan} 
\end{align}
In view of \eqref{E:vtan}, we get that the vector $v=(v'_1,v''_2,\dots,v'_n,v''_n)$
 lies in $T_pM$, and the equations in \eqref{E:ivtan} ensure that $iv\in T_pM$ 
  (viewing $v=(v'_1+iv''_1, \dots, v'_n+iv''_n)\in \cplx^n$). 
  This is a contradiction to the fact that $M$ is totally real at $p$. 
  
  For the converse, assume the matrix $A_p$ has rank $n$. We show that $M$ is 
  totally real at $p\in M$. Suppose $M$ is not totally real at $p$, i.e.,
  there exists a $v\in T_pM$, $v\neq 0$, such that $iv\in T_pM$. This implies that equations \eqref{E:vtan} and 
  \eqref{E:ivtan} hold. Hence, 
   $A_pv=0$,
which contradicts the fact that rank of $A_p$ is $n$. Hence, $M$ is totally real at $p$.
\end{proof} 
 
Next we state a lemma that gives a characterization  for a graph in $\cplx^{2n}$, using the 
graphing functions, to have complex tangents.
\begin{lemma}\label{L:totreal}
  Let $f^1,\dots, f^n: \cplx^n \lrarw \cplx$ be $\smoo^1$-smooth functions. Let
$M:= \{(z,w)\in \cplx^{2n} : w_\nu=f^{\nu}(z), \; \nu=1, \dots, n\}$. Let $P:=(a,f^1(a), \dots, f^n(a)) \in M$. 
Then, $M$ has a complex tangent at $P\in M$ if and only if there exists a vector  
$(v_1, \dots, v_n)\in \cplx^n \setminus \{0\}$ such that
\[
 \sum_{j=1}^n v_j {\frac{\bdy f^\nu}{\bdy \ba{z_j}}(a)}=0 \quad \forall \nu=1, \dots, n.
\]
\end{lemma}

\begin{proof}
The proof follows from the following fact due to Wermer \cite{W3}.
\smallskip

\noindent{\bf Fact.} {\em Let $h_1, \dots, h_m$ be $\smoo^1$-smooth complex valued functions defined 
in a neighbourhood $U$ of $0\in \rea^k$ such that the function $h:=(h_1,\dots, h_m)$ 
is a regular map on $U$ into $\cplx^m$. Let $S:=h(U)$. Then, $S$ is totally real at 
$h(0)$ if and only if the complex rank of the matrix $\begin{pmatrix}
                                                       {\dfrac{\bdy h_j}{\bdy x_k}}(0)
                                                      \end{pmatrix}_{i,j}$ 
is $k$.}

\end{proof}

  \section{The proof of Theorem~\ref{T:totreal_PCVX}}\label{proof:totreal_PCVX}
  
  We begin the proof by constructing a tube-like neighbourhood of the graph and a 
  non-negative plurisubharmonic function defined in it, which vanishes on the graph. 
  This constitutes Step I. Further steps then lead us to showing the desired compact 
  to be polynomially convex on the basis of our construction in the first step.
  \smallskip
  
  {\bf Step I:}{\em Constructing a tube-like neighbourhood $G$ of the graph and a 
  plurisubharmonic function $u$ on $G$.}  
  \smallskip
  
In this case we consider the defining functions:
\[
u_j(z,w)=\begin{cases} 
                   \rl w_j-\rl f_j(z),\; \text{if}\; $j=1,\dots, n$\\
                   \imag w_j -\imag f_j(z),\;\text{if}\; $j=n+1,\dots,2n$.
                   \end{cases}
                   \] and
\[
u(z,w):=\sum_{j=1}^n u_j^2(z,w).
\]
We obtain the Levi form:                  
\begin{align}
 \mathscr{L}u(\bcdot;V) 
&= \sum_{j=1}^n \left((f_j-w_j)\sum_{k,l=1}^n \ba{\bdy^2_{\zbar_k z_l}f_jv_l \ba{v_k}} + (\ba{f_j}- \ba{w_j})
\sum_{k,l=1}^n \bdy^2_{z_k \zbar_l}f^jv_k \ba{v_l} \right) \notag\\
& \qquad + \sum_{k,l=1}^n \left( \sum_j^n \bdy_{z_k}f_j \ba{\bdy_{z_l}f_j}\right)v_k \ba{v_l}
+ \sum_{k,l=1}^n \left( \sum_{j=1}^n \bdy_{\zbar_l}f_j\ba{\bdy_{\zbar_k}f_j} \right) v_k \ba{v_l} \notag\\
& \qquad - \sum_{k,l=1}^n \bdy_{z_k}f_l v_k \ba{t_l}- \sum_{k,l=1}^n \ba{ \bdy_{z_l}f_k v_l \ba{t_k}} +\sum_{j=1}^n |t_j|^2, \notag
\end{align}
where we denote $V=(v,t)=(v_1,\dots,v_n,t_1,\dots,t_n) \in \cplx^{2n}$. 
Swapping the subscripts $j$ and $k$ in the first sum in the second line above allows us to see that:
\begin{align}
 \mathscr{L}u (\bcdot;V) &= 2\sum_{\nu=1}^n \rl{ \left( (\ba{f^\nu}- \ba{w_\nu}) \mathfrak{L}f^{\nu}(z;v) \right)}
+\sum_{j,k=1}^n \left( \sum_{\nu=1}^n f^\nu_{z_j} \ba{f^\nu_{z_k}} \right) v_j \ba{v_k} \notag\\
& \qquad + \sum_{j,k=1}^n \left( \sum_{\nu=1}^n f^\nu_{\zbar_k} \ba{f^\nu_{\zbar_j}} \right) v_j \ba{v_k}
- \sum_{j,k=1}^n f^k_{z_j}v_j \ba{t_k}- \sum_{j,k=1}^n \ba{ f^j_{z_k}v_k \ba{t_j}} +\sum_{\nu=1}^n |t_{\nu}|^2. \notag\\
&= 2\sum_{\nu=1}^n \rl{ \left( (\ba{f^\nu}- \ba{w_\nu}) \mathfrak{L}f^{\nu}(z;v) \right)}
+ \sum_{\nu=1}^n|f^\nu_{z_1}v_1+....+f^\nu_{z_n}v_n-t_\nu|^2 \notag\\
&\qquad \qquad\qquad\qquad\qquad \qquad\qquad\qquad
+\sum_{j,k=1}^n \left( \sum_{\nu=1}^n f^\nu_{\zbar_k} \ba{f^\nu_{\zbar_j}} \right) v_j \ba{v_k} \notag\\
 &= 2\sum_{\nu=1}^n \rl{ \left( (\ba{f^\nu}- \ba{w_\nu}) \mathfrak{L}f^{\nu}(z;v) \right)}
+ \sum_{\nu=1}^n|f^\nu_{z_1}v_1+....+f^\nu_{z_n}v_n-t_\nu|^2 \notag\\
&\qquad \qquad\qquad\qquad\qquad \qquad\qquad\qquad
+\sum_{\nu=1}^n|f^\nu _{\ba{z_1}} \ba{v_1}+ \dots +f^\nu_{\ba{z_n}} \ba{v_n}|^2 \notag\\
& \geq \sum_{\nu=1}^n|f^\nu _{\ba{z_1}} \ba{v_1}+ \dots +f^\nu_{\ba{z_n}} \ba{v_n}|^2 
  - 2\sum_{\nu=1}^n  |f^\nu- w_\nu| |\mathfrak{L}f^{\nu}(z;v)| \label{E:barderivative}.
\end{align}
Let 
\[
L(z):= \max_{\nu}\left( \sup_{||v||=1}  |\mathfrak{L}f^{\nu}(z;v)|  \right) ,
\]
and
\[
 m(z):=\inf_{||v||=1}\left( \sum_{\nu=1}^n|f^\nu _{\ba{z_1}}(z) \ba{v_1}+....+f^\nu_{\ba{z_n}}(z) \ba{v_n}|^2  \right).
\]
Since the graph $S$ is  totally real, by Lemma~\ref{L:totreal}, we have
$m(z)>0$ for all $z \in \cplx^n$. Define
\[
 G:= \left\lbrace (z,w)\in \cplx^{2n}: \sum_{\nu=1}^n|f^\nu(z)-w_\nu|< \frac{m(z)}{2 L(z)} \right\rbrace.
\]
 From \eqref{E:barderivative}, it is clear that $u$ is strictly plurisubharmonic on $G$ and
$u^{-1}\{0\}={\sf Gr}(f^1,...,f^n)$. Since $\om \subset G$ (by hypothesis), we have $u \in {\sf psh}(\om)$
and $K\subset {\sf Gr}_{\cplx^n}(F) \subset u^{-1}\{0\}$.
\smallskip

\noindent {\bf{Step II}}: {\em Showing that $\hull{K} \subset u^{-1} \{0\}$}.
\smallskip

\noindent Since, by Step I,  $u \in {\sf psh}(\om)$ and $K \subset u^{-1}\{0\}$,
all the conditions of Lemma \ref{P:extPsh} are fulfilled with given compact $K$, 
$v:= u$ and $\phi:=\Psi$. Hence, in view of Lemma~\ref{P:extPsh}, we obtain
\[
\hull{K} \subset u^{-1} \{0\}.
\]
\smallskip
 
\noindent {\bf{Step III}}: {\em Completing the proof}.
\smallskip

\noindent The aim of this step is to show that 
$K$ is polynomially convex. For that we consider  $K_1:=\{(z,w)\in{\sf Gr}_{\cplx^n}(F) : \Psi(z,w)+\eps\leq 0\}$,
where
\[
 -\eps:= \sup_K \Psi(z,w).
\]
Clearly, $K\subset K_1$.
Thanks to the fact that $K\subset \om=\{(z,w)\in \cplx^n \times \cplx^n : \Psi(z,w)<0\}$, we get that $\eps>0$.  
$\Psi$ is plurisubharmonic in $\cplx^{2n}$, 
\[
\hull{K_1}\subset \{(z,w)\in \cplx^{2n}: \Psi(z,w)<0\}=\om\subset G .
\] 
By Lemma~\ref{P:extPsh}, with the compact $K_1$, $\Omega:=G$
$v:=u$ and $\phi:=\Psi$, we
conclude that $\hull{K_1}\subset u^{-1}\{0\}=M$. Hence, $K_1$ is polynomially convex.  
Using Lemma~\ref{L:totrlsub}, we conclude that $K$ is polynomially convex.

\smallskip

The converse follows from Result~\ref{R:polyhed}.

\section{Proof of Theorem~\ref{T:totrlpoly}}

Our proof of Theorem~\ref{T:totrlpoly} follows in lines similar to that of Theorem~\ref{T:totreal_PCVX}. 
Again, using the defining equations,  we will construct a nonnegative plurisubharmonic function in a 
tubular neighbourhood of the given compact subset $K$. 

\begin{proof}[Proof of Theorem~\ref{T:totrlpoly}]
As before, we divide the proof in three steps. 
\smallskip

\noindent {\bf Step I.} {\em Existence of a plurisubharmonic function $u$ on a neighbourhood of K with $K\subset u^{-1}\{0\}$.}
 
Let us define the following function: 
\[
 u(z):= \sum_{j=1}^{2n-k}\rho_j^2(z).
\]
We now compute the Levi-form for the above function $u$. For that, We have 
\begin{equation}\label{E:2ndder}
 \dfrac{\bdy^2 u}{\bdy z_j\bdy \ba{z_k}}(z)
 =2\sum_{l=1}^{2n-k}\rho_l(p)\dfrac{\bdy^2\rho_l}{\bdy z_j\bdy \ba{z_k}}(z) 
 +2\sum_{l=1}^{2n-k} \dfrac{\bdy \rho_l}{\bdy z_j}(p) \dfrac{\bdy \rho_l}{\bdy\ba{z_k}}(z). 
\end{equation}

Hence, the Levi-form of $u$: 
\begin{align}
 \mathscr{L}u(z, v) &= \sum_{j,k=1}^n \dfrac{\bdy^2\rho_p}{\bdy z_j\bdy \ba{z_k}}(z) v_j\ba{v_k} \notag\\
 &= 2\sum_{l=1}^{2n-k}\sum_{j,k=1}^n\dfrac{\bdy^2\rho_l}{\bdy z_j\bdy \ba{z_k}}(z) v_j\ba{v_k} 
 + 2 \sum_{l=1}^{2n-k} \left|\sum_{j=1}^n\dfrac{\bdy \rho_l}{\bdy \ba{z_j}}(z)v_j\right|^2 \notag \\
 &= 2 \sum_{l=1}^{2n-k}\rho_l(z)\mathscr{L}\rho_l(z, v)
 + 2 \sum_{l=1}^{2n-k} \left|\sum_{j=1}^n\dfrac{\bdy \rho_l}{\bdy \ba{z_j}}(z)v_j\right|^2 \notag\\
 & \geq 2\sum_{l=1}^{2n-k} \left|\sum_{j=1}^n\dfrac{\bdy \rho_l}{\bdy \ba{z_j}}(z)v_j\right|^2 
 -2 \sum_{l=1}^{2n-k}|\rho_l(z)||\mathscr{L}\rho_l(z, v)|. \label{E:leviformrho}
\end{align}

Let us define the following set
\[
 \OM:= \left\lbrace z\in \cplx^n : \sum_{l=1}^{2n-k}|\rho_l(z)|<\dfrac{m(z)}{L(z)} \right\rbrace,
\]
where
\[
L(z):=\max_{l} \left(\sup_{||v||=1}|\mathscr{L}\rho_l(z, v)|\right),
\]
and
\[
 m(z):=\inf_{||v||=1}\sum_{l=1}^{2n-k} \left|\sum_{j=1}^n\dfrac{\bdy \rho_l}{\bdy \ba{z_j}}(z)v_j\right|^2.
\]
Since $M$ is totally real, by Lemma~\ref{L:chartotrl}, we get that $m(z)>0$ for $z\in M$. 
From \eqref{E:leviformrho}, we obtain that 
\[
\mathscr{L} u(z, v) \geq 0,\; \text{for all}\; z\in \Omega,
\]

Hence, $u$ is plurisubharmonic in $\OM$ and $K\subset u^{-1}\{0\}$. 
\smallskip

\noindent {\bf Step II.} {\em Showing that $\hull{K}\subset u^{-1}\{0\}$}.
\smallskip

Let us denote $\om:=\{z\in \cplx^n : \phi(z)<0\}$. By the assumption $u$ is plurisubharmonic in 
$\om$. Invoking Lemma~\ref{P:extPsh} again with $\OM:=\om$, $v:=u$, we get that
\[
\hull{K}\subset u^{-1}\{0\}.
\]
\noindent {\bf Step III.} {\em Completing the proof}.
\smallskip

As in the proof of Theorem~\ref{T:totreal_PCVX} we consider
\[
K_1:=M\cap \{z\in \cplx^n : \phi(z)+\eps\leq 0\},
\]
where $-\eps=\sup_{K}\phi(z)$. The remaining part of the proof goes in the same way as in 
Step III of the proof of Theorem~\ref{T:totreal_PCVX}.

As before, the converse follows from Result~\ref{R:polyhed}
\end{proof}
\section {Examples}\label{S:example}
In this section we provide a couple of examples of totally-real submanifolds of $\cplx^2$: 
the first one is given by H\"{o}rmander-Wermer \cite{HW}. 
\begin{example}
We consider the graph
$K=\{(z,f(z))\in\cplx^2: z\in\ba{\dsc}\}$
over the closed unit disc $\ba{\dsc}$, where $f(z)=-(1+i)\ba{z}+iz\ba{z}^2+z^2\ba{z}^3$. It is shown in \cite{HW}, 
by attaching an analytic disc to $\{(z,f(z)): |z|=1\}\subset K$, that $K$ is not polynomially convex. 
Here we focus on the closed subsets of the compact $K$ of the form:
\[
K_r:=\{(z,f(z))\in \cplx^2: |z|\leq r\}.
\] 
Since $K$ is a subset of a totally-real submanifold $M=\{(z,f(z)):z\in\cplx\}$, we already know 
that there exists an $r>0$ such that $\hull{K_r}=K_r$. Here a range for $r$ is deduced for which $K_r$ 
is polynomially convex.
We use Theorem~\ref{T:totreal_PCVX} to show that $K_r$ is 
polynomially convex for all $r\in [0, 1/\sqrt{3}]$.
\smallskip

Let us first compute:
\begin{align}
\dfrac{\partial f}{\partial \ba{z}}(z)&= -(1+i)+2i|z|^2+3|z|^4, \notag\\
\dfrac{\partial^2	f}{\partial z\partial \ba{z}}(z)&=2\ba{z}(i+3|z|^2). \notag
\end{align}
Hence, in the notation of Theorem~\ref{T:totreal_PCVX}, we have 
\begin{align}
L(z)&= \left |\dfrac{\partial^2 f}{\partial z\partial\ba{z}}(z)\right|=2|z|\sqrt{1+9|z|^4} \notag\\
m(z)&=\left|\dfrac{\partial f}{\partial\ba{z}}(z)\right|^2=9|z|^8-2|z|^4-4|z|^2+2. \label{E:modFzbar} 
\end{align} 
We get a neighbourhood of $K$ as 
\[
\Omega:=\left\{(z,w)\in\cplx^2:|w-f(z)|<\dfrac{9|z|^8-2|z|^4-4|z|^2+2}{2|z|\sqrt{1+9|z|^4}}\right\},
\]
in which $u(z,w):=|w-f(z)|^2$ is plurisubharmonic.
We now note that the function $h(r):=9r^8-2r^4-4r^2+2$ is monotonically decreasing in 
the interval $[0,1/\sqrt{3}]$. Hence, from \eqref{E:modFzbar} we have 
\begin{equation*}
\inf_{|z|\leq 1/\sqrt{3}} m(z)= 1.
\end{equation*} 
We also have 
\[
\sup_{|z|\leq 1/\sqrt{3}}L(z)=\dfrac{4}{3\sqrt[4]{3}}.
\]
Hence, we get that the open set $\{(z,w)\in\cplx^2:|z|<1/\sqrt{3}, |w-f(z)|<3\sqrt[4]{3}/4\}\subset \Omega$. 
The implies that
\[
\{(z,w)\in\cplx^2: |z|<1/\sqrt{3}, |w|<|f(z)|+3\sqrt[4]{3}/4\}\subset \Omega.
\]
Since $\sup_{|z|\leq 1/\sqrt{3}}|f(z)|<3\sqrt[4]{3}/4$, therefore, we obtain:
\[
K_r\Subset D(0;1/\sqrt{3})\times D(0;3\sqrt[4]{3}/4) \subset \Omega,\qquad \forall r\in [0,1/\sqrt{3}].
\]
Since an open bidisc in $\cplx^2$ is always a sub-level set of some plurisubharmonic function 
on $\cplx^2$, therefore, by using Theorem~\ref{T:totreal_PCVX}, 
we conclude that $\hull{K_r}=K_r$ for all $r\in [0,1/\sqrt{3}]$.
\end{example}
\smallskip

\begin{example}
Let us consider the following graph in $\CC$ over $\rea^2$:
\[
M=\{(x_1+ic(x_1^2+x_2^3), x_2+id(x_2^2+x_1^3))\in\cplx^2: x_1,x_2\in\rea\},
\]
where $0\leq c, d\leq\dfrac{1}{20}$.
We show that the compact $K:=M\cap\ba{\dsc^2}$ is polynomially convex.
In this case we have $\rho:=(\rho_1,\rho_2)$, where
\begin{align*}
\rho_1(z_1,z_2)&:=\dfrac{1}{2i}(z_1-\ba{z_1})-\dfrac{c}{4}\left((z_1+\ba{z_1})^2+\dfrac{1}{2}(z_2+\ba{z_2})^3\right) \\
\rho_2(z_1,z_2) &:=\dfrac{1}{2i}(z_2-\ba{z_2})-\dfrac{d}{4}\left(\dfrac{1}{2}(z_2+\ba{z_2})^2+(z_1+\ba{z_1})^3\right).
\end{align*}
Clearly $K=\rho^{-1}\{0\}\cap\ba{\dsc^2}$. Using the notation $z_j=x_j+iy_j$, $j=1,2$, we compute:
\begin{align*}
\dfrac{\partial\rho_1}{\partial\ba{z_1}}(z)&=i/2-cx_1,\qquad
\dfrac{\partial\rho_2}{\partial\ba{z_2}}(z)=i/2-dx_2,\\
\dfrac{\partial\rho_1}{\partial\ba{z_2}}(z)&= -\dfrac{3cx_2^2}{2},\qquad
\dfrac{\partial\rho_2}{\partial\ba{z_1}}(z)=-\dfrac{3dx_1^2}{2},\\
\dfrac{\partial^2\rho_1}{\partial z_1\partial\ba{z_1}}(z)&= -c/2,\qquad 
\dfrac{\partial^2\rho_1}{\partial z_1\partial\ba{z_2}}(z)=0,\qquad
\dfrac{\partial^2\rho_1}{\partial z_2\partial\ba{z_2}}(z)=-\dfrac{3cx_2}{2},\\
\dfrac{\partial^2\rho_2}{\partial z_1\partial\ba{z_1}}(z)&= -\dfrac{3dx_1}{2}, \qquad
\dfrac{\partial^2\rho_2}{\partial z_1\partial\ba{z_2}}(z)=0,\qquad
\dfrac{\partial^2\rho_2}{\partial z_2\partial\ba{z_2}}(z)=-d/2.
\end{align*}
From the above computation we get that:
\begin{align*}
L(z)&=\max_{l=1,2}\left( \sup_{||v||=1}\left|\dfrac{\partial^2\rho_l}{\partial z_1 \partial \ba{z_1}}(z)|v_1|^2 
+2\rl \left(\dfrac{\partial^2\rho_l}{\partial z_1 \partial \ba{z_2}}(z)v_1\ba{v_2}\right) + 
\dfrac{\partial^2\rho_l}{\partial z_2 \partial \ba{z_2}}(z)|v_2|^2\right|\right)\\
&\leq 2 \max\{c, d\}\\
m(z)& \geq \dfrac{1}{4} +\inf_{||v||=1}\left[ c^2\left(x_1v_1+\dfrac{3}{2} x_2^2v_2\right)^2
+d^2\left(x_2v_2+\dfrac{3}{2}x_1^2v_1\right)^2\right].
\end{align*}
Hence, the neighbourhood 
\[
\Omega=\left\{z\in\cplx^2: |\rho_1(z)|+|\rho_2(z)|<\dfrac{m(z)}{L(z)}\right\}
\]
of $K$ contains the open set 
\[
\left\{z\in\cplx^n: |y_1-c(x_1^2+x_2^3)|+|y_2-d(x_2^2+x_1^3)|<\dfrac{1}{8\max\{c,d\}}\right\},
\]
and $u(z):=\rho_1^2(z)+\rho_2^2(z)$ is plurisubharmonic in $\Omega$.
We now note that 
\[
K\Subset D(0;1+\varepsilon)\times D(0;1+\varepsilon)\subset\Omega,
\]
where $0<\varepsilon<1/20$.
Therefore, by applying Theorem~\ref{T:totrlpoly}, we conclude that $K$ is polynomially convex.
\end{example}

\noindent {\bf Acknowledgements.} I would like thank Gautam Bharali for the long discussions 
that we had on an earlier version of Theorem~\ref{T:totreal_PCVX} during my Ph.D. 
years. I would like to thank Stat-Math Unit of Indian Statistical Institute, Bangalore Centre 
for providing active support and excellent facilities.

\end{document}